\documentclass[11pt,centertags]{amsart}
\usepackage{amssymb}
\usepackage{hyperref}
\usepackage{enumitem}
\usepackage{url}
\usepackage{gloss}
\usepackage[pdftex]{graphicx}
    \DeclareGraphicsExtensions{.eps,.pdf,.png,.jpg}
    \graphicspath{{}}
\usepackage{mathtools}
\usepackage[marginratio=1:1,height=8.5in,width=6.5in,tmargin=1.25in]{geometry}
\usepackage{comment}
\usepackage[normalem]{ulem}
\usepackage[usenames,dvipsnames]{color}
	
\usepackage{todonotes}

\allowdisplaybreaks

\usepackage[T1]{fontenc}
\usepackage[utf8]{inputenc}
\usepackage{textcomp}
\usepackage{bbm}
\usepackage{mathpazo}
\usepackage{mathrsfs}

\newtheorem*{main*}{Main Theorem}
\newtheorem{main}{Theorem}
\newtheorem{maincor}{Corollary}

\newtheorem{theorem}{Theorem}[section]
\newtheorem*{theorem*}{Theorem}
\newtheorem{proposition}[theorem]{Proposition}
\newtheorem{lemma}[theorem]{Lemma}
\newtheorem{corollary}[theorem]{Corollary}

\newtheorem*{question*}{Question}
\newtheorem{conjecture}{Conjecture}
\newtheorem*{conjecture*}{Conjecture}
\newtheorem{notation}[theorem]{Notations}

\newtheorem*{fact*}{Fact}

\theoremstyle{definition}

\newtheorem*{definition*}{Definition}

\theoremstyle{remark}
\newtheorem{remark}[theorem]{Remark}

\numberwithin{equation}{section}

\newcommand{\cout}[1]{}
\definecolor{darkcyan}{rgb}{0. 0.65, 0.65}

\newcommand{\R}{\mathbb{R}}

\DeclareMathOperator{\tr}{tr}

\newcommand{\Ric}{\mathrm{Ric}}
\newcommand{\Pf}{\mathrm{Pf}}
\newcommand{\cF}{\mathcal{F}}

\newcommand{\dvol}{dvol}
\newcommand{\del}{\partial }

\providecommand{\to}{\longrightarrow }

\newcommand{\eq}[1]{\begin{align}\begin{split} #1 \end{split}\end{align}}
\newcommand{\eqn}[1]{\begin{align*}\begin{split} #1 \end{split}\end{align*}}

\newcommand{\matiii}[3]{\left(\begin{array}{ccc} #1\\#2\\#3\end{array}\right)}

\def\be#1\ee{\begin{align}\begin{split} #1 \end{split}\end{align}}
\def\beq#1\eeq{\begin{align*}\begin{split} #1 \end{split}\end{align*}}

\renewcommand{\geq}{\geqslant}
\renewcommand{\leq}{\leqslant}

\begin{document}

\author[Chris Connell]{Chris Connell$^\dagger$}
\thanks{$\dagger$ Supported in part by Simons Foundation grant \#965245.}

\author[Yuping Ruan]{Yuping Ruan}

\author[Shi Wang]{Shi Wang$^\ddagger$}
\thanks{$\ddagger$ Supported in part by the National Natural Science Foundation of China}

\title{Nonpositively curved $4$-manifolds with zero Euler characteristic}

\address{Indiana University}
\email{connell@indiana.edu}

\address{ShanghaiTech University}
\email{wangshi@shanghaitech.edu.cn, shiwang.math@gmail.com}

\address{Northwestern University}
\email{ruanyp@northwestern.edu, ruanyp@umich.edu}

\begin{abstract}
We show that for any closed nonpositively curved Riemannian 4-manifold $M$ with vanishing Euler characteristic, the Ricci curvature must degenerate somewhere. Moreover, for each point $p\in M$, either the Ricci tensor degenerates or else there is a foliation by totally geodesic flat 3-manifolds in a neighborhood of $p$. As a corollary, we show that if in addition the metric is analytic, then the universal cover of $M$ has a nontrivial Euclidean de Rham factor. Finally we discuss how this result creates an implication of conjectures on simplicial volume in dimension four.
\end{abstract}
\maketitle

\thispagestyle{empty}

\section{Introduction}

The well-known Gauss-Bonnet theorem states that the Euler characteristic of a surface (up to a constant multiple $1/2\pi$) equals the integral of the Gaussian curvature. This was generalized by Chern \cite{Che44,Che45} to all even dimensions, which is now known as the Gauss-Bonnet-Chern Theorem and states,
\[\chi(M)=\frac{1}{(2\pi)^n}\int_M \Pf_g(x) dvol_g(x),\]
where $(M,g)$ is a closed oriented Riemannian $2n$-manifold, and $\Pf_g(x)dvol_g$ is the Pfaffian of the curvature form on $M$. One significance is that it closely relates the intrinsic geometry and the underlying topology of a given smooth manifold. An immediate consequence of the formula, as first noticed by John Milnor \cite{Che55}, is that when $n=2$ and $M^4$ admits a nonpositively curved metric, since $\Pf(x)$ does not change sign, the signed Euler characteristic satisfies $(-1)^n\chi(M)\geq 0$, which answers in the affirmative a conjecture of Heinz Hopf \cite{Hopf32} in dimension $4$. However, Geroch \cite{Geroch76} showed that the same strategy cannot be played in dimension $6$ or higher. The conjecture remains open for $n>2$.

The main purpose of this paper is to further analyze in dimension $4$ the equality case $\chi(M)=0$. We give a full geometric description in terms of the local metric structure and the Ricci tensor. In the real analytic case, using a result of Schroeder \cite{Sch89}, we further show that the Riemannian universal cover $\widetilde M$ of $M$ must split off an $\R$-factor. More precisely, we show the following.

\begin{main}\label{main thm}
Let $(M,g)$ be a 4-dimensional closed Riemannian manifold with non-positive curvature and Euler characteristic $\chi(M)=0$. Then the Ricci curvature $\Ric$ must be degenerate somewhere in $M$. Moreover, for any $p\in M$, either of the following holds.
\begin{enumerate}
\item[(a).] $\Ric$ is degenerate at $p$;
\item[(b).] There exists a neighborhood $U_p$ of $p$ such that $U_p$ is isometric to $(U,g_f)$, where $U$ an open subset of $\R^4$,
$$g_f|_{(x_1,x_2,x_3,x_4)}=dx_1^2+dx_2^2+dx_3^2+(f(x_1,x_2,x_3,x_4))^2dx_4^2$$
and $f:U\to\R_+$ is a smooth function such that $f(\cdot,\cdot,\cdot,x_4)$ is strictly convex whenever it is defined. In particular, for any $c\in\R$, $\{(x_1,x_2,x_3,x_4)|x_4=c\}\cap U$ is totally geodesic and flat under the metric $g_f$ whenever it is non-empty.
\end{enumerate}
\end{main}

\begin{maincor}\label{main cor}
If $(M,g)$ is a real-analytic 4-dimensional closed Riemannian manifold with non-positive curvature and Euler characteristic $\chi(M)=0$, then $\widetilde M$ has a nontrivial Euclidean de Rham factor.
\end{maincor}
\begin{proof}[Proof of Corollary \ref{main cor} from Theorem \ref{main thm}] If not, by \cite[Corollary 2]{GZ11}, the Ricci curvature of $(M,g)$ must be negative definite somewhere. By the real analytic assumption and Theorem \ref{main thm}, there exist uncountably many totally geodesic flats of dimension three in the universal cover $(\widetilde M, g)$ of $(M,g)$. By Schroeder's classification result \cite[Theorem 2]{Sch89} of maximal higher rank submanifolds in analytic 4-manifolds of nonpositive curvature, $(\widetilde M,g)$ cannot be rank one and therefore $\mathrm{rank}(\widetilde M,g)\geq 2$. It follows from the celebrated rank rigidity theorem (proved by Ballmann in \cite{Bal85} and Burns-Spatzier in \cite{BS87a} and \cite{BS87b} independently and with different methods) that $(\widetilde M,g)$ is a Riemannian product of two nonpositively curved Riemannian manifolds. Let $(\widetilde M,g)=(M_1,g_1)\times (M_2,g_2)$ with $(M_j,g_j)$ nonpositively curved. Since there exist some $p=(p_1,p_2)$ such that the Ricci curvature $\Ric_g|_p$ is negative definite, $\dim(M_1)=\dim(M_2)=2$ and the sectional curvature $K_{(M_j,g_j)}(p_j)$ of $(M_j,g_j)$ at $p_j$ is negative for any $j\in\{1,2\}$. In particular, there does not exist a totally geodesic and flat submanifold of dimension three passing through $p$. This contradicts Theorem \ref{main thm}.
\end{proof}

\textbf{Plan of the proof:} The proof of Theorem \ref{main thm} is based on the following three observations:
\begin{enumerate}
\item[(1).] If the Ricci curvature is nondegenerate (equivalently, negative definite) at a point $p$, then there exist a unique one dimensional subspace $\R v\subset T_pM$ for some unit vector $v\in T_pM$ such that for any $w,u\in T_pM$ with $w,u\perp v$, we have $R(w,u,w,u)=0$. See Lemma \ref{prelim local structure} in Section \ref{s2};
\item[(2).] One can use the Gauss-Bonnet-Chern formula \eqref{GBC formula} to compute the Euler characteristic $\chi(M)=\chi(M,g)$ of $(M,g)$, which is independent of the choice of the Riemannian metric $g$. Since $\chi(M)=0$ and $M$ is nonpositively curved, the integrand $\mathrm{Pf}_g$ in the Gauss-Bonnet-Chern formula \eqref{GBC formula} vanishes everywhere. Use the aforementioned one dimensional subspace, we construct a smooth path of Riemannian metrics $g^{(s)}=g+sq$ for some symmetric $2$-tensor supported near $p$. Then we compute the first variation of $\chi(M)$ (as a constant function on the space of Riemannian metric) along $g^{(s)}$ when $s=0$ via the Gauss-Bonnet-Chern formula \eqref{GBC formula}, which equals $0$. One can see that $\mathrm{Pf}_g\equiv 0$ greatly reduces the complexity of this computation. Based on this first variation computation, we show that the orthogonal complement of the aforementioned one dimensional subspaces form a smooth integrable distribution near $p$. Moreover, the integral manifolds of this distribution are totally geodesic and flat. See Lemma \ref{integrable+totally geo flats} in Section \ref{s3};
\item[(3).] If the Ricci curvature is non-degenerate everywhere, the aforementioned one dimensional subspaces are defined everywhere and hence form a smooth real line bundle. Up to a two-fold cover, one can find a unit vector field $V$ as a section of this bundle. Then we can show that $\mathrm{div}\nabla_VV=\Ric_g(V,V)$. It follows from the divergence theorem that $\Ric_g(V,V)\equiv 0$, forcing the Riemannian metric $g$ to be flat. This contradicts the assumption that Ricci curvature is non-degenerate everywhere. See Section \ref{s4} for more details.
\end{enumerate}
In Section \ref{s2}, we also discuss some standard curvature formulas and computations. The proof of Theorem \ref{main thm} is given in Section \ref{s4}. Some applications of our result to simplicial volume can be found in Section \ref{s5}.

\subsection*{Acknowledgments} We would like to thank Livio Flaminio for a careful conversation on curvature computations. We also thank Ben Schmidt and Ralf Spatzier for helpful discussions. 

\section{Preliminaries}\label{s2}
\begin{notation}
Let $(M,g)$ be a Riemannian manifold with Levi-Civita connection $\nabla$, {which is determined by the following Koszul formula:
\begin{align}\label{connection formula}
2\langle \nabla_XY,Z\rangle=X\langle Y,Z\rangle+Y\langle X,Z\rangle-Z\langle X,Y\rangle+\langle[X,Y],Z\rangle-\langle[X,Z],Y\rangle-\langle[Y,Z],X\rangle.
\end{align}
for any vector fields $X,Y,Z$ defined on some open subset of $(M,g)$.} Throughout this paper, we adopt the following convention (e.g. \cite{DoCarmo92} p.89-91) on Riemannian curvature tensor: For any vector fields $X,Y,Z,W$, we define the curvature tensor as:
$$R(X,Y)Z:=\nabla_Y\nabla_XZ-\nabla_X\nabla_YZ+\nabla_{[X,Y]}Z\mathrm{~and~}R(X,Y,Z,W)=g(R(X,Y)Z,W).$$
{For any point $p\in M$ and any vectors $X,Y,Z,W,U\in T_pM$}, the Riemannian curvature tensor has the following properties.
\begin{enumerate}
\item[(a).] Skew symmetries and symmetries:
\eq{\label{curv sym}R(X,Y,Z,W)=-R(Y,X,Z,W)=-R(X,Y,W,Z)=R(Z,W,X,Y);}
\item[(b).] The first Bianchi identity:
\eq{\label{1st Bianchi} R(X,Y,Z,W)+R(Y,Z,X,W)+R(Z,X,Y,W)=0;}
\item[(c).] The second Bianchi identity:
\eq{\label{2nd Bianchi} \nabla R(X,Y,Z,W;U)+\nabla R(X,Y,W,U;Z)+\nabla R(X,Y,U,Z;W)=0.}
\end{enumerate}
\end{notation}
\subsection{The Gauss-Bonnet-Chern formula in dimension 4}\label{ss21}
Denoted by $|\cdot|_g$ the norm of a tensor induced by the Riemannian metric $g$. Recall that we have the following Gauss-Bonnet-Chern formula in dimension 4.
\begin{theorem}[{The Gauss-Bonnet-Chern formula in dimension 4, \cite{Che44}}]
Let $(M,g)$ be a closed oriented Riemannian $4$-manifold. The \emph{Pfaffian} function of $(M,g)$ is given by $\Pf_g:=(|R|_g^2-4|\Ric|_g^2+R_g^2)/8$, where $R_g=\tr_g\Ric$ is the scalar curvature function of $(M,g)$. Denoted by $\chi(M)$ the Euler charateristic of $M$. Then we have 
\eq{\label{GBC formula}
\chi(M)=\frac{1}{(2\pi)^2}\int_M\Pf_g(p)\dvol_g(p)=\frac{1}{32\pi^2}\int_M(|R|_g^2-4|\Ric|_g^2+R_g^2)(p)\dvol_g(p).
}
\end{theorem}
Fix any point $p\in M$ and any orthonormal basis $\{v_1,v_2,v_3,v_4\}$ in $T_pM$. Denoted by $R_{ijkl}=R(v_i,v_j,v_k,v_l)$ for simplicity, where $1\leq i,j,k,l\leq 4$. Then we have
\eq{\label{Pf general}
\Pf_g(p)=&R_{1212}R_{3434}+R_{1213}R_{3442}+R_{1214}R_{3423}+R_{1223}R_{3414}+R_{1224}R_{3431}+R_{1234}R_{3412} \\
&+R_{1312}R_{4234}+R_{1313}R_{4242}+R_{1314}R_{4223}+R_{1323}R_{4214}+R_{1324}R_{4231}+R_{1334}R_{4212} \\
&+R_{1412}R_{2334}+R_{1413}R_{2342}+R_{1414}R_{2323}+R_{1423}R_{2314}+R_{1424}R_{2331}+R_{1434}R_{2312}\\
=&R_{1212}R_{3434}+R_{1313}R_{4242}+R_{1414}R_{2323}+R_{1234}^2+R_{1342}^2+R_{1423}^2 \\
&-2R_{1213}R_{4243}-2R_{1214}R_{3234}-2R_{1314}R_{2324}-2R_{2123}R_{4143}-2R_{2124}R_{3134}-2R_{3132}R_{4142}.
}
In particular, if we further assume that $v_1,v_2,$ and $v_3$ are chosen such that
\eq{\label{ON basis cond}
R(v_1,v_2,v_1,v_2)=\min_{v\perp w, v,w\in T^1_pM}R(v,w,v,w)~\mathrm{and}~R(v_1,v_3,v_1,v_3)=\min_{\begin{subarray}~v\in\R v_1\oplus\R v_2,\\ w\perp\R v_1\oplus\R v_2,\\ v,w\in T^1_pM\end{subarray}}R(v,w,v,w),
}
then $R_{1213}=R_{1214}=R_{2123}=R_{2124}=R_{3132}=R_{1314}=0$. Therefore the Pfaffian function at $p$ can be simplified as (see also \cite[Page 125]{Che55})
\eq{\label{Pf simplified}
\Pf_g(p)=R_{1212}R_{3434}+R_{1313}R_{4242}+R_{1414}R_{2323}+R_{1234}^2+R_{1342}^2+R_{1423}^2.
}
For non-positively curved Riemannian manifolds with zero Euler characteristic, we have the following lemma.
\begin{lemma}\label{prelim local structure}
Let $(M,g)$ be a non-positively curved closed $4$-manifold with Euler characteristic $\chi(M)=0$. For any $p\in M$ and any orthonormal frame $v_1,v_2,v_3,v_4\in T_pM$ satisfying condition \eqref{ON basis cond}, either of the following holds.
\begin{enumerate}
\item[(a).] $\Ric$ is degenerate at $p$;
\item[(b).] There exist an open neighborhood $U_p$ of $p$ and a smooth unit vector field $V$ defined on $U_p$ satisfying the following properties:
\begin{enumerate}
\item[(i).] $\Ric$ is negative definite on $U_p$;
\item[(ii).] $V$ is a simple eigenvector of $\Ric$ with the smallest eigenvalue $R_g/2$;
\item[(iii).] $V|_p=v_1$.
\end{enumerate}
\end{enumerate}
\end{lemma}
\begin{proof}
Since $(M,g)$ is non-positively curved, \eqref{Pf simplified} implies that the Pfaffian function $\Pf_g\geq 0$. By the Gauss-Bonnet-Chern formula \eqref{GBC formula}, we have $\Pf_g\equiv0$. Denoted by $R_{ijkl}=R(v_i,v_j,v_k,v_l)$ for simplicity, where $1\leq i,j,k,l\leq 4$. By \eqref{Pf simplified} we have $R_{1212}R_{3434}=R_{1313}R_{4242}=R_{1414}R_{2323}=0$. 

If $R_{1212}=0$, by the first assumption in \eqref{ON basis cond}, we have $R_{1313}=R_{1414}=0$, which implies that $\Ric_g(v_1,v_1)=0$. Hence $\Ric$ is degenerate at $p$. 

If $R_{1212}<0$, then $R_{3434}=0$. By the second assumption in \eqref{ON basis cond}, $0\geq R_{4242}\geq R_{1313}$, and $0\geq R_{1414}\geq R_{1313}$. Therefore $R_{4242}=0$ and $R_{1414}R_{2323}=0$. If $R_{1414}=0$, then $\Ric_g(v_4,v_4)=0$, which implies that $\Ric$ is degenerate at $p$. 

Now we assume that $\Ric$ is negative definite at $p$. By the previous discussions, we have $0>R_{1414}\geq R_{1313}\geq R_{1212}$. Hence, $R_{2323}=R_{2424}=R_{3434}=0$. Therefore $R_g/2=R_{1212}+R_{1313}+R_{1414}+R_{2323}+R_{2424}+R_{3434}=\Ric_g(v_1,v_1)$. Hence it remains to show that $v_1$ is a simple eigenvector of $\Ric$ with the smallest eigenvalue. Observing that $0$ is the largest possible sectional curvature, we also have $R_{2324}=R_{3234}=R_{4243}=0$ and  
$R_{1424}=R_{1434}=R_{1343}=0$. Recall from condition \eqref{ON basis cond}, equation \eqref{Pf simplified} and the fact that $\Pf_g(p)=0$, we already have
$$R_{1213}=R_{1214}=R_{2123}=R_{2124}=R_{3132}=R_{1314}=R_{1234}=R_{1342}=R_{1423}=0.$$
Therefore for any unit vector $v=\sum_{i=1}^4a_iv_i\in T_pM$, we have 
$$\Ric_g(v,v)=a_1^2(R_{1212}+R_{1313}+R_{1414})+\sum_{i=2}^4a_i^2R_{i1i1}\geq (R_{1212}+R_{1313}+R_{1414})=\Ric_g(v_1,v_1)=\frac{R_g}{2}.$$ 
Equality holds only when $|a_1|=1$ and $a_2=a_3=a_4=0$. Hence $v_1$ is a simple eigenvector of $\Ric$ with the smallest eigenvalue $R_g/2$. Since $U:=\{q\in M: \Ric \mathrm{~is~negative~definite}\}=\{q\in M: \det_g(\Ric)\neq 0\}$ is open, the smallest eigenvalue of $\Ric$ is simple and equal to $R_g/2$ in a neighborhood of $p$. Then the existence of an open neighborhood of $U_p$ satisfying property (i) and a smooth unit vector field $V$ satisfying properties (ii) and (iii) follows directly from the Implicit Function Theorem.
\end{proof}

We will also need to use the following lemma later.

\begin{lemma}\label{lem:zeroRic_implies_zeroPf}
	Let $(M,g)$ be a non-positively curved, closed $4$-manifold. If $\Ric$ is degenerate at $p$, then $\Pf_g(p)=0$.
\end{lemma}
\begin{proof}
	We choose an orthonormal basis $\{v_1,v_2,v_3,v_4\}$ at $p$ such that $\Ric_g(v_1,v_1)=0$. We adopt a similar notation and write $R_{ijkl}=R(v_i,v_j,v_k,v_l)$ for $1\leq i,j,k,l\leq 4$. By the non-positive curvature assumption, we have $R(v_1,u,v_1,u)=0$ for any $u\perp v_1$. Since $R(v_1,u,v_1,u)=0$ achieves the maximal curvature, we have $R_{1213}=R_{1214}=R_{1314}=R_{2123}=R_{2124}=R_{3132}=R_{3134}=R_{4142}=R_{4143}=0$. In particular, $R(v_1,u,v_2,u)=0$ for any $u$ in the span of $\{v_3, v_4\}$. Now take $u_t=\cos t \cdot v_3+\sin t\cdot v_4$ and use
	\[\left.\frac{d}{dt}\right\vert_{t=0}R(v_1,u_t,v_2,u_t)=0,\]
	we obtain that $R_{1324}+R_{1423}=0$. Now switching the role of $v_2, v_3$ and run the same argument, we obtain that $R_{1234}+R_{1432}=0$. It follows that $R_{1432}=-R_{1423}=R_{1324}=R_{3142}$ and $R_{1432}=-R_{1234}=R_{4312}$. Finally use the first Bianchi identity that $R_{1432}+R_{3142}+R_{4312}=0$, we see that $R_{1432}=R_{3142}=R_{4312}=0$. Thus, in view of equation \eqref{Pf general}, we have that $\Pf(p)=0$.
\end{proof}
\subsection{Riemannian curvature tensors under metric variations}\label{ss22}
Let $q(\cdot,\cdot)$ be a symmetric $2$-tensor on $(M,g)$. Let $g^{(s)}=g+sq$ be a smooth family of Riemannian metrics defined for $|s|\ll 1$. Denoted by $R^{(s)}$ the corresponding Levi-Civita connection and Riemannian curvature tensor for $g^{(s)}$. Then we have the following first variation formula for $R^{(s)}$. For simplicity, we also use the notations $\langle\cdot,\cdot\rangle$ for $g(\cdot,\cdot)=g^{(0)}(\cdot,\cdot)$ and $R$ for the Riemannian curvature tensor of $g$.
\begin{proposition}[{\cite[1.174 Theorem, page 62]{Bes07}}]\label{curvature 1st var}
Under the above assumptions, for any vector fields $X,Y,Z,W$ on $M$, we have
\begin{align*}
&\left\langle\left.\frac{d}{ds}\right|_{s=0}R^{(s)}(X,Y)Z,W\right\rangle \\
=&\frac{1}{2}[\nabla^2q(X,W;Z;Y)-\nabla^2q(X,Z;W;Y) 
-\nabla^2q(Y,W;Z;X)+\nabla^2q(Y,Z;W;X)] \\
&-\frac{1}{2}[q(R(X,Y)Z,W)+q(Z,R(X,Y)W)].
\end{align*}
and
\eq{\label{(0,4)-curv var}
&\left.\frac{d}{ds}\right|_{s=0}R^{(s)}(X,Y,Z,W) \\
=&\frac{1}{2}[\nabla^2q(X,W;Z;Y)-\nabla^2q(X,Z;W;Y) 
-\nabla^2q(Y,W;Z;X)+\nabla^2q(Y,Z;W;X)] \\
&+\frac{1}{2}[q(R(X,Y)Z,W)-q(Z,R(X,Y)W)].
}
\end{proposition}
\begin{remark}
Our notations are different from the notations in \cite{Bes07}. In particular, we use $\nabla q(\cdot,\cdot;X)$ instead of $\nabla_Xq(\cdot,\cdot)$ and $\nabla^2 q(\cdot,\cdot;Y;X)$ instead of $\nabla^2_{X,Y}q(\cdot,\cdot)$. (See \cite[1.17, page 25]{Bes07}.) One can check that our convention on Riemannian curvature tensors coincide with \cite[equation (1.21), page 26]{Bes07}.
\end{remark}

We want to apply the above proposition to a special case when $q(\cdot,\cdot)=\langle\cdot, V\rangle\langle\cdot, V\rangle$ for some smooth vector field $V$ on $M$. For any vector fields $X,Y,Z,W$ on $M$, the first and second covariant derivatives of $q$ are computed as follows. 
\eqn{
\nabla q(X,Y;Z)=&Zq(X,Y)-q(\nabla_ZX,Y)-q(X,\nabla_ZY) \\
=&Z(\langle X, V\rangle\langle Y , V\rangle)-\langle\nabla_ZX, V\rangle\langle Y, V\rangle-\langle X, V\rangle\langle\nabla_ZY, V\rangle  \\
=&\langle X, \nabla_ZV\rangle\langle Y, V\rangle+\langle X, V\rangle\langle Y, \nabla_ZV\rangle
}
and 
\eq{\label{nabla2q}
&\nabla^2 q(X,Y;Z;W)\\
=&W\nabla q(X,Y;Z)-\nabla q(\nabla_WX,Y;Z)-\nabla q(X,\nabla_WY;Z)-\nabla q(X,Y;\nabla_WZ) \\
=&W(\langle X, \nabla_ZV\rangle\langle Y, V\rangle+\langle X, V\rangle\langle Y, \nabla_ZV\rangle)\\
&-\langle\nabla_WX,\nabla_ZV\rangle\langle Y,V\rangle
-\langle\nabla_WX,V\rangle\langle Y,\nabla_ZV\rangle \\
&-\langle X,\nabla_ZV\rangle\langle\nabla_WY,V\rangle-\langle X,V\rangle\langle\nabla_WY,\nabla_ZV\rangle \\
&-\langle X,\nabla_{\nabla_WZ}V\rangle\langle Y,V\rangle-\langle X,V\rangle\langle Y,\nabla_{\nabla_WZ}V\rangle\\
=&\langle X,(\nabla_W\nabla_Z-\nabla_{\nabla_WZ})V\rangle\langle Y,V\rangle+\langle X,V\rangle\langle Y,(\nabla_W\nabla_Z-\nabla_{\nabla_WZ})V\rangle\\
&+\langle X,\nabla_ZV\rangle\langle Y,\nabla_WV\rangle+\langle X,\nabla_WV\rangle\langle Y,\nabla_ZV\rangle.
}
Equations \eqref{(0,4)-curv var} and \eqref{nabla2q} yield
\eq{\label{(0,4)-curv var special q}
&2\left.\frac{d}{ds}\right|_{s=0} R^{(s)}(X,Y,Z,W) \\
=&\nabla^2q(X,W;Z;Y)-\nabla^2q(X,Z;W;Y) 
-\nabla^2q(Y,W;Z;X)+\nabla^2q(Y,Z;W;X) \\
&+q(R(X,Y)Z,W)-q(Z,R(X,Y)W)  \\
=&\langle X,(\nabla_Y\nabla_Z-\nabla_{\nabla_YZ})V\rangle\langle W,V\rangle+\langle X,V\rangle\langle W,(\nabla_Y\nabla_Z-\nabla_{\nabla_YZ})V\rangle \\
&-\langle X,(\nabla_Y\nabla_W-\nabla_{\nabla_YW})V\rangle\langle Z,V\rangle-\langle X,V\rangle\langle Z,(\nabla_Y\nabla_W-\nabla_{\nabla_YW})V\rangle \\
&-\langle Y,(\nabla_X\nabla_Z-\nabla_{\nabla_XZ})V\rangle\langle W,V\rangle-\langle Y,V\rangle\langle W,(\nabla_X\nabla_Z-\nabla_{\nabla_XZ})V\rangle \\
&+\langle Y,(\nabla_X\nabla_W-\nabla_{\nabla_XW})V\rangle\langle Z,V\rangle+\langle Y,V\rangle\langle Z,(\nabla_X\nabla_W-\nabla_{\nabla_XW})V\rangle \\
&+\langle X,\nabla_ZV\rangle\langle W,\nabla_YV\rangle+\langle X,\nabla_YV\rangle\langle W,\nabla_ZV\rangle  \\
&-\langle X,\nabla_WV\rangle\langle Z,\nabla_YV\rangle-\langle X,\nabla_YV\rangle\langle Z,\nabla_WV\rangle  \\
&-\langle Y,\nabla_ZV\rangle\langle W,\nabla_XV\rangle-\langle Y,\nabla_XV\rangle\langle W,\nabla_ZV\rangle  \\
&+\langle Y,\nabla_WV\rangle\langle Z,\nabla_XV\rangle+\langle Y,\nabla_XV\rangle\langle Z,\nabla_WV\rangle  \\ 
&+\langle R(X,Y)Z,V\rangle\langle W,V\rangle+\langle Z,V\rangle\langle R(X,Y)W,V\rangle.
}
{(Here, the $j$-th and $(j+4)$-th line after the second equality sign corresponds to the $j$-th term after the first equality sign for any $1\leq j\leq 4$.)} In particular, if $Y=W$ and $X,Y,Z\perp V$ with respect to $g$, we have
\eq{\label{main ingredient for Pf var}
&2\left.\frac{d}{ds}\right|_{s=0}\langle R^{(s)}(X,Y)Z,Y\rangle_s \\
=&\langle X,\nabla_ZV\rangle\langle Y,\nabla_YV\rangle+\langle X,\nabla_YV\rangle\langle Y,\nabla_ZV\rangle  \\
&-\langle X,\nabla_YV\rangle\langle Z,\nabla_YV\rangle-\langle X,\nabla_YV\rangle\langle Z,\nabla_YV\rangle  \\
&-\langle Y,\nabla_ZV\rangle\langle Y,\nabla_XV\rangle-\langle Y,\nabla_XV\rangle\langle Y,\nabla_ZV\rangle  \\
&+\langle Y,\nabla_YV\rangle\langle Z,\nabla_XV\rangle+\langle Y,\nabla_XV\rangle\langle Z,\nabla_YV\rangle.  \\ 
}

\subsection{Curvature computation of Riemannian metrics of the form $g_1+f^2g_2$}
Let $(M_j,g_j)$ be Riemannian manifolds, $j=1,2$, and $f:M_1\times M_2\to\R_+$ be a smooth function. In this subsection only, we denote by $\langle\cdot,\cdot\rangle$, $\langle\cdot,\cdot\rangle_1$ and $\langle\cdot,\cdot\rangle_2$ the inner product induced by $g_1+f^2g_2$, $g_1$ and $g_2$ respectively.

For any vector fields $X_j, Y_j, Z_j$ on $M_j$, let $X=X_1+X_2$, $Y=Y_1+Y_2$, $Z=Z_1+Z_2$ be vector fields on $M$. (Here, if we denote by  $\pi_j:M\to M_j$ the natural projection maps for $j=1,2$, then the vector field $X_1:M_1\times M_2\to T(M_1\times M_2)$ on $M=M_1\times M_2$ is defined as $X_1(p_1,p_2)=(X_1(p_1),0)$. Similarly $X_2(p_1,p_2)=(0,X_2(p_2))$.) Let $\nabla$, $\nabla^{(1)}$, $\nabla^{(2)}$ be the Levi-Civita connections on $(M,g_1+f^2g_2)$, $(M_1,g_1)$ and $(M_2,g_2)$ respectively. Then {by \eqref{connection formula},} we have
\begin{align*}
&2\langle \nabla_XY,Z\rangle\\
=&X\langle Y,Z\rangle+Y\langle X,Z\rangle-Z\langle X,Y\rangle+\langle[X,Y],Z\rangle-\langle[X,Z],Y\rangle-\langle[Y,Z],X\rangle\\
=&X_1\langle Y_1,Z_1\rangle_1+Y_1\langle X_1,Z_1\rangle_1-Z_1\langle X_1,Y_1\rangle_1+\langle[X_1,Y_1],Z_1\rangle_1-\langle[X_1,Z_1],Y_1\rangle_1-\langle[Y_1,Z_1],X_1\rangle_1 \\
&+[X_1(f^2)+X_2(f^2)]\langle Y_2,Z_2\rangle_2+[Y_1(f^2)+Y_2(f^2)]\langle X_2,Z_2\rangle_2-[Z_1(f^2)+Z_2(f^2)]\langle X_2,Y_2\rangle_2 \\
&+f^2(X_2\langle Y_2,Z_2\rangle_2+Y_2\langle X_2,Z_2\rangle_2-Z_2\langle X_2,Y_2\rangle_2+\langle[X_2,Y_2],Z_2\rangle_2-\langle[X_2,Z_2],Y_2\rangle_2-\langle[Y_2,Z_2],X_2\rangle_2)\\
=&2\langle \nabla^{(1)}_{X_1}Y_1,Z_1\rangle_1+2f^2\langle \nabla^{(2)}_{X_2}Y_2,Z_2\rangle_2\\
&+2f (X_1f+X_2f)\langle Y_2,Z_2\rangle_2+2f (Y_1f+Y_2f)\langle X_2,Z_2\rangle_2-2f (Z_1f+Z_2f)\langle X_2,Y_2\rangle_2.
\end{align*}
In particular
\begin{align*}
\langle\nabla_{X_1}Y_1,Z_1\rangle=&\langle\nabla^{(1)}_{X_1}Y_1,Z_1\rangle_1=\langle\nabla^{(1)}_{X_1}Y_1,Z_1\rangle;\\
\langle\nabla_{X_1}Y_1,Z_2\rangle=&\langle\nabla_{X_2}Y_1,Z_1\rangle=\langle\nabla_{X_1}Y_2,Z_1\rangle=0
;\\
\langle\nabla_{X_2}Y_1,Z_2\rangle=&f (Y_1f)\langle X_2,Z_2\rangle_2= \frac{Y_1f}{f}\langle X_2,Z_2\rangle;\\
\langle\nabla_{X_1}Y_2,Z_2\rangle=&f (X_1f)\langle Y_2,Z_2\rangle_2=\frac{X_1f}{f}\langle Y_2,Z_2\rangle;\\
\langle\nabla_{X_2}Y_2,Z_1\rangle=&-f (Z_1f)\langle X_2,Y_2\rangle_2=-\frac{Z_1f}{f}\langle X_2,Y_2\rangle;\\
\langle\nabla_{X_2}Y_2,Z_2\rangle=&f^2\langle \nabla^{(2)}_{X_2}Y_2,Z_2\rangle_2+f (X_2f)\langle Y_2,Z_2\rangle_2+f (Y_2f)\langle X_2,Z_2\rangle_2-f( Z_2f)\langle X_2,Y_2\rangle_2\\
=&\langle \nabla^{(2)}_{X_2}Y_2,Z_2\rangle+ \frac{X_2f}{f}\langle Y_2,Z_2\rangle+\frac{ Y_2f}{f}\langle X_2,Z_2\rangle- \frac{Z_2f}{f}\langle X_2,Y_2\rangle.
\end{align*}
Hence $\nabla_{X_1}X_1\in TM_1\times\{0\}\subset TM$ and $\nabla_{X_1}Y_2\in\{0\}\times TM_2\subset TM$. Since $[X_1,Y_2]=0$, we have
\begin{align}\label{convexity}
\begin{split}
&R(X_1,Y_2,X_1,Y_2)\\
=&\langle\nabla_{Y_2}\nabla_{X_1}X_1-\nabla_{X_1}\nabla_{Y_2}X_1, Y_2\rangle\\
=&Y_2\langle\nabla_{X_1}X_1, Y_2\rangle-\langle\nabla_{X_1}X_1, \nabla_{Y_2}Y_2\rangle-\langle\nabla_{X_1}\nabla_{Y_2}X_1, Y_2\rangle\\
=&\frac{(\nabla_{X_1}X_1)f}{f}\langle Y_2,Y_2\rangle-\langle\nabla_{X_1}\nabla_{X_1}Y_2,Y_2\rangle\\
=&\frac{(\nabla_{X_1}X_1)f}{f}\langle Y_2,Y_2\rangle-X_1\langle\nabla_{X_1}Y_2,Y_2\rangle+\langle\nabla_{X_1}Y_2,\nabla_{X_1}Y_2\rangle\\
=&\frac{(\nabla_{X_1}X_1)f}{f}\langle Y_2,Y_2\rangle-X_1\left(f(X_1f)\langle Y_2,Y_2\rangle_2\right)+\frac{X_1f}{f}\langle \nabla_{X_1}Y_2,Y_2\rangle\\
=&\frac{(\nabla_{X_1}X_1)f}{f}\langle Y_2,Y_2\rangle-\left(\frac{f(X_1X_1f)+(X_1f)^2}{f^2}\right)\langle Y_2,Y_2\rangle+\left(\frac{X_1f}{f}\right)^2\langle Y_2,Y_2\rangle\\
=&\frac{(\nabla_{X_1}X_1)f}{f}\langle Y_2,Y_2\rangle-\frac{X_1X_1f}{f}\langle Y_2,Y_2\rangle=-\frac{\mathrm{Hess}_gf(X_1,X_1)}{f}\langle Y_2,Y_2\rangle.
\end{split}
\end{align}

\section{Codimension one, totally geodesic and flat local foliations near points with non-degenerate Ricci curvature}\label{s3}
Let $(M,g)$ be a non-positively curved, closed oriented $4$-manifold with $\chi(M)=0$. Recall that in Lemma \ref{prelim local structure}, we have for any point $p\in M$ with negative definite $\Ric$, there exist a unit vector field $V$ defined in a neighborhood $U_p$ of $p$ such that $2\Ric_g(V,V)=R_g$. In particular, $V$ satisfies the following two properties (as a direct corollary of Lemma \ref{prelim local structure}).
\begin{enumerate}
\item[(iv).] For any $q\in U_p$ and any non-zero vectors $v\perp w\in \{u\in T_qM:u\perp V\}$, the sectional curvature between $v$ and $w$ is zero;
\item[(v).] For any $q\in U_p$ and any non-zero vector $w\perp V$, the sectional curvature between $w$ and $V$ is negative.
\end{enumerate}
{To see why the above two properties holds, for any unit vectors $v,w\in T_qM$ such that $v,w\perp V$ and $v\perp w$, we choose an orthonormal basis $v_1=V|_q,v_2=w,v_3=v,v_4$ in $T_qM$. Then 
$$\sum_{2\leq i,j\leq 4}R(v_i,v_j,v_i,v_j)=R_g|_q-\Ric_g(v_1,v_1)-\sum_{i=2}^4R(v_i,v_1,v_i,v_1)=R_g|_q-2(\Ric_g)|_q(V,V)=0.$$
Since $(M,g)$ is nonpositively curved, $R(v_i,v_j,v_i,v_j)=0$ for any $2\leq i,j\leq 4$. In particular, $R(w,v,w,v)=R(v_2,v_3,v_2,v_3)=0$, which proves property (iv). Moreover, $\Ric_g(w,w)=R(v_1,w,v_1,w)=R(V|_q,w,V|_q,w)$. Property (v) then follows from the assumption that $\Ric$ is negative definite in $U_p\ni q$.
}

The following lemma proves a stronger version of property (iv).
\begin{lemma}\label{integrable+totally geo flats}
Under the above assumptions, the orthogonal complement $V^\perp$ is an integrable distribution. Moreover, its integral manifolds (defined locally) are totally geodesic and flat.
\end{lemma}
\begin{proof}
It suffices to show that $\langle\nabla_\cdot V,\cdot\rangle$ vanish on $V^\perp$. {(This is because for any locally defined smooth vector fields $X,Y$ tangent to the distribution $V^\perp$ (equivalently, $X,Y$ are perpendicular to $V$ whenever they are defined), we have
$$\langle [X,Y],V\rangle=\langle\nabla_XY-\nabla_YX,V\rangle=X\langle Y,V\rangle-\langle Y,\nabla_XV\rangle-(Y\langle X,V\rangle-\langle X,\nabla_YV\rangle)=0.$$
Therefore $[X,Y]\perp V$ whenever they are defined. This proves that $V^\perp$ is integrable. Hence $\langle\nabla_\cdot V,\cdot\rangle|_{V^\perp}$ is the second fundamental form of the integral manifolds of $V^\perp$ up to a sign. $\langle\nabla_\cdot V,\cdot\rangle|_{V^\perp}\equiv 0$ is equivalent to the integral manifolds of $V^\perp$ being totally geodesic. Flatness of these integral manifolds follows from the above property (iv).)} We split the proof into two steps. 

\textbf{Step 1: $\langle\nabla_\cdot V,\cdot\rangle$ is skew symmetric on $V^\perp$.} This first step essentially follows from the Bianchi identities. For any $q\in U_p$ and any non-zero vectors $X,Y,Z\in T_qM$ perpendicular to $V$, we can extend them to vector fields also denoted by $X,Y,Z$ defined in an open neighborhood $O_q\subset U_p$ of $q$ such that $X,Y,Z\perp V$ on $V_q$. Then by property (iv) above, we have
\eqn{
R(X,Y,X,\cdot)=R(Y,X,Y,\cdot)\equiv 0\mathrm{~on~}O_q \quad (\mathrm{0~is~the~largest~sectional~curvature.})
}
and
\eq{\label{eq0}
R(X,Y,Z,\cdot)+R(Z,Y,X,\cdot)\equiv 0\mathrm{~on~}O_q. \quad \left(\begin{subarray}~\mathrm{Substitute}~X~\mathrm{by}~X+tZ~\\\mathrm{in~above~and~differentiate~at~}t=0.\end{subarray}\right)
}
Then the first Bianchi identity \eqref{1st Bianchi}, the skew symmetry in the first two entries of the Riemannian curvature tensor \eqref{curv sym} and the above \eqref{eq0} imply that
\eqn{
2R(Y,X,Z,\cdot)=&-R(X,Y,Z,\cdot)+R(Y,X,Z,\cdot)~\:\,\, \quad\quad\quad\quad\quad\quad(\mathrm{\eqref{curv sym},~skew~symmetry~in~}X,Y.)\\
=&R(Z,Y,X,\cdot)-R(Z,X,Y,\cdot)\quad\quad\quad\quad\quad\quad\quad\quad(\mathrm{\eqref{eq0}~applied~to~both~terms.})\\
=&R(Z,Y,X,\cdot)+R(X,Z,Y,\cdot)\quad\quad\quad\quad\quad\quad\quad\quad(\mathrm{\eqref{curv sym},~skew~symmetry~in~}X,Z.)\\
=&-R(Y,X,Z,\cdot)\implies R(Y,X,Z,\cdot)=0\mathrm{~on~}O_q.~ (\mathrm{\eqref{1st Bianchi},~the~first~Bianchi~identity.})
}
By arbitariness of $Z$ perpendicular to $V$ and \eqref{curv sym}, we have
\eq{\label{eq1}
R(X,Y,\cdot,\cdot)=R(\cdot,\cdot,X,Y)\equiv 0\mathrm{~on~}O_q.
} 
Hence by \eqref{eq1} and arbitariness of $X,Y$ perpendicular to $V$, we have
\eq{\label{eq2}
\nabla R(X,Y,X,Y;V)=VR(X,Y,X,Y)-2R(\nabla_VX,Y,X,Y)-2R(X,\nabla_VY,X,Y)=0,
} 
\eq{\label{eq3}
\nabla R(X,Y,Y,V;X)=&XR(X,Y,Y,V)-R(\nabla_XX,Y,Y,V)\\
&-R(X,\nabla_XY,Y,V)-R(X,Y,\nabla_XY,V)-R(X,Y,Y,\nabla_XV)\\
=&-R(V,Y,Y,V)\langle\nabla_XX,V\rangle-R(X,V,Y,V)\langle\nabla_XY,V\rangle\\
=&-R(V,Y,Y,V)(X\langle X,V\rangle-\langle X,\nabla_XV\rangle) \\
&-R(X,V,Y,V)(X\langle Y,V\rangle-\langle Y,\nabla_XV\rangle)\\
=&R(V,Y,Y,V)\langle X,\nabla_XV\rangle+R(X,V,Y,V)\langle Y,\nabla_XV\rangle
}
and similarly
\eq{\label{eq4}
\nabla R(X,Y,V,X;Y)=\nabla R(Y,X,X,V;Y)=R(V,X,X,V)\langle Y,\nabla_YV\rangle+R(Y,V,X,V)\langle X,\nabla_YV\rangle.
}
Apply the second Bianchi identity \eqref{2nd Bianchi} to \eqref{eq2}, \eqref{eq3} and \eqref{eq4}, we have
\eq{\label{eq after Bianchi}
0=&\nabla R(X,Y,X,Y;V)+\nabla R(X,Y,Y,V;X)+\nabla R(X,Y,V,X;Y)\\
=&R(V,Y,Y,V)\langle X,\nabla_XV\rangle+R(X,V,Y,V)\langle Y,\nabla_XV\rangle \\
&+R(V,X,X,V)\langle Y,\nabla_YV\rangle+R(Y,V,X,V)\langle X,\nabla_YV\rangle \\
=&R(X,V,Y,V)(\langle Y,\nabla_XV\rangle+\langle X,\nabla_YV\rangle)-R(Y,V,Y,V)\langle X,\nabla_XV\rangle-R(X,V,X,V)\langle Y,\nabla_YV\rangle.
}
Denoted by $V^\perp|_q:=\{u\in T_qM:u\perp V\}$. For simplicity, we write $A(u,w)=-R(u,V,w,V)$ and $B(u,w)=(\langle\nabla_uV,w\rangle+\langle\nabla_wV,u\rangle)/2$ for any $u,w\in V^\perp|_q$. Then by property (v), $A$ is positive definite and $B$ is symmetric on $V^\perp|_q$. For any $X,Y\in V^\perp|_q$, \eqref{eq after Bianchi} can be simplified as
\eq{\label{simple eq after Bianchi}
A(X,X)B(Y,Y)+A(Y,Y)B(X,X)-2A(X,Y)B(X,Y)=0.
}
Let $u_1,u_2,u_3$ be an orthonormal eigenbasis for $A$ on $V^\perp|_q$ (i.e. $u_i$ are unit eigenvectors of $A$ on $V^\perp|_q$ and $u_1,u_2,u_3$ are orthogonal to each other). For simplicity we write $A_{ij}=A(u_i,u_j)$ and $B_{ij}=B(u_i,u_j)$. Then $A_{ij}=0$ whenever $i\neq j$ and $A_{ii}>0$, following the fact that $A$ is positive definite. For any $\theta\in\R$, we let $X=u_1$ and $Y=u_2\cos\theta+u_3\sin\theta$. Then $A(X,Y)=0$ and \eqref{simple eq after Bianchi} yields
\eq{\label{eq theta}
A_{11}(B_{22}\cos^2\theta +B_{33}\sin^2\theta +2B_{23}\sin\theta\cos\theta )+(A_{22}\cos^2\theta +A_{33}\sin^2\theta )B_{11}=0.
}
Subtracting $\eqref{eq theta}|_{\theta=0}\cdot\cos^2\theta+\eqref{eq theta}|_{\theta=\pi/2}\cdot\sin^2\theta$ from \eqref{eq theta}, we have
$$2A_{11}B_{23}\sin\theta\cos\theta=0.$$
Since $A_{11}>0$, $B_{23}=0$. Similarly, $B_{12}=B_{23}=0$, which implies that $A$ and $B$ are simultaneously diagonalizable. Without loss of generality, we assume that $B_{11}B_{22}\geq 0$. Then when $\theta=0$, \eqref{eq theta} implies that
$$0=(A_{11}B_{22}+A_{22}B_{11})^2=A_{11}^2B_{22}^2+A_{22}^2B_{11}^2+2A_{11}A_{22}B_{11}B_{22}.$$
Since $A_{11},A_{22}>0$, we have $B_{11}=B_{22}=0.$ Hence by \eqref{eq theta}, $B_{33}=0$. This finishes the proof of step 1 since $B\equiv0$ is equivalent to $\langle\nabla_\cdot V,\cdot\rangle$ being skew symmetric on $V^\perp$.

\textbf{Step 2: $\langle\nabla_\cdot V,\cdot\rangle$ vanish on $V^\perp$.}
This last step essentially follows from variation of the Gauss-Bonnet-Chern formula along a certain smooth path of Riemannian metrics. For any $q\in U_p$, let $O_q\subset U_p$ be an open neighborhood satisfying the following properties:
\begin{enumerate}
\item[(a).] There exists vector fields $V=V_1,V_2,V_3,V_4$ on $O_q$ such that they form an orthonormal basis with respect to the metric $g$ in every fiber of $TO_q$;
\item[(b).] There exists a bump function $\rho:M\to \R$ properly supported on $O_q$. That is to say, $\rho$ is a smooth function satisfying the following properties.
\begin{enumerate}
\item[(i).] $0\leq \rho(x)\leq 1$ for any $x\in M$;
\item[(ii).] There exists open neighborhoods $O_q'\subset \overline{O_q'}\subset O_q''\subset \overline{O_q''}\subset O_q$ of $q$ such that $\rho(x)=1$ for any $x\in O_q'$ and $\rho(x)=0$ for any $x\in O_q\setminus \overline{O_q''}$.
\end{enumerate}
\end{enumerate}
Let $q(\cdot,\cdot)|_x=\langle\cdot,\rho(x)V\rangle\langle\cdot,\rho(x)V\rangle$ for any $x\in M$. Then $\rho(x)V$ is a smooth vector field defined on $M$ and $q$ is a smooth symmetric 2-tensor on $M$. Let $g^{(s)}=g+sq$ for $|s|\ll 1$. Then $V_1/(1+s\rho^2)^{1/2},V_2,V_3,V_4$ form an orthonormal basis with respect to the metric $g^{(s)}$ in every fiber of $TO_q$. Denoted by $R$, $R^{(s)}$ the Riemannian curvature tensor of $g$, $g^{(s)}$ and $R_{ijkl}=R(V_i,V_j,V_k,V_l),R^{(s)}_{ijkl}=R^{(s)}(V_i,V_j,V_k,V_l),\forall~1\leq i,j,k,l\leq 4$, respectively. Hence by \eqref{Pf general} we have $\Pf_{g^{(s)}}(x)=\Pf_g(x)$ for any $x\in M\setminus O_q$ and
\eq{\label{Pf_s}
&\Pf_{g^{(s)}}(x)\left[1+s(\rho(x))^2\right] \\
=&R^{(s)}_{1212}R^{(s)}_{3434}+R^{(s)}_{1313}R^{(s)}_{4242}+R^{(s)}_{1414}R^{(s)}_{2323}+\left(R^{(s)}_{1234}\right)^2+\left(R^{(s)}_{1342}\right)^2+\left(R^{(s)}_{1423}\right)^2 \\
&-2R^{(s)}_{1213}R^{(s)}_{4243}-2R^{(s)}_{1214}R^{(s)}_{3234}-2R^{(s)}_{1314}R^{(s)}_{2324}-2R^{(s)}_{2123}R^{(s)}_{4143}-2R^{(s)}_{2124}R^{(s)}_{3134}-2R^{(s)}_{3132}R^{(s)}_{4142}
}
for any $x\in O_q$. When $s=0$, by \eqref{eq1} and the assumption that $V_1=V,V_2,V_3,V_4$ are orthonormal with respect to $g=g^{(0)}$ in every fiber, we have $R_{ijkl}=0$ whenever $1\not\in\{i,j\}\cap\{k,l\}$. In particular.
$0=R_{3434}=R_{4242}=R_{2323}=R_{1234}=R_{1342}=R_{1423}=R_{4243}=R_{3234}=R_{2324}=R_{2123}=R_{4143}=R_{2124}=R_{3134}=R_{3132}=R_{4142}$ in $O_q$. Therefore doing the first variation of \eqref{Pf_s} when $s=0$ yields
\eq{\label{var of Pf}
\left.\frac{d}{ds}\right|_{s=0}\Pf_{g^{(s)}}(x)=&R_{1212}\left.\frac{d}{ds}\right|_{s=0}R^{(s)}_{3434}+R_{1313}\left.\frac{d}{ds}\right|_{s=0}R^{(s)}_{4242}+R_{1414}\left.\frac{d}{ds}\right|_{s=0}R^{(s)}_{2323}\\
&-2R_{1213}\left.\frac{d}{ds}\right|_{s=0}R^{(s)}_{4243}-2R_{1214}\left.\frac{d}{ds}\right|_{s=0}R^{(s)}_{3234}-2R_{1314}\left.\frac{d}{ds}\right|_{s=0}R^{(s)}_{2324}.
}
For simplicity, we denote by $Q_{ij}:=\langle V_i,\nabla_{V_j}V\rangle$ in $O_q$ for any $2\leq i,j,\leq 4$. Then by Step 1 of the proof, we have $Q_{ij}=-Q_{ji}$. Moreover, for any $x\in O_q$ and any $X,Y\in T_xM$ perpendicular to $V$, we have
\eq{\label{eq with rho}
\langle X,\nabla_{Y}(\rho(x)V)\rangle=\rho(x)\langle X,\nabla_{Y}V\rangle+Y\rho(x)\langle X,V\rangle=\rho(x)\langle X,\nabla_{Y}V\rangle.
}
In particular, $\langle \cdot,\nabla_{\cdot}(\rho(x)V)\rangle$ is also skew-symmetric on $V^\perp$ by Step 1. By \eqref{eq with rho} and \eqref{main ingredient for Pf var}, for any $x\in O_q$ and any $X,Y,Z\in O_q$ perpendicular to $V$, we have
\eq{\label{simplified curv var}
&\left.\frac{d}{ds}\right|_{s=0}\langle R^{(s)}(X,Y)Z,Y\rangle_s \left(=\left.\frac{d}{ds}\right|_{s=0}\langle R^{(s)}(Y,X)Y,Z\rangle_s\right)\\
=&\frac{1}{2}\left(\langle X,\nabla_Z(\rho(x)V)\rangle\langle Y,\nabla_Y(\rho(x)V)\rangle+\langle X,\nabla_Y(\rho(x)V)\rangle\langle Y,\nabla_Z(\rho(x)V)\rangle\right.  \\
&\quad-\langle X,\nabla_Y(\rho(x)V)\rangle\langle Z,\nabla_Y(\rho(x)V)\rangle-\langle X,\nabla_Y(\rho(x)V)\rangle\langle Z,\nabla_Y(\rho(x)V)\rangle  \\
&\quad-\langle Y,\nabla_Z(\rho(x)V)\rangle\langle Y,\nabla_X(\rho(x)V)\rangle-\langle Y,\nabla_X(\rho(x)V)\rangle\langle Y,\nabla_Z(\rho(x)V)\rangle  \\
&\quad\left.+\langle Y,\nabla_Y(\rho(x)V)\rangle\langle Z,\nabla_X(\rho(x)V)\rangle+\langle Y,\nabla_X(\rho(x)V)\rangle\langle Z,\nabla_Y(\rho(x)V)\rangle\right)\quad \\ 
=&3\langle X,\nabla_Y(\rho(x)V)\rangle\langle Y,\nabla_Z(\rho(x)V)\rangle \\
=&3(\rho(x))^2\langle X,\nabla_YV\rangle\langle Y,\nabla_ZV\rangle(=3(\rho(x))^2\langle Y,\nabla_XV\rangle\langle Z,\nabla_YV\rangle),
}
where the second equality follows from the skew-symmetry of $\langle \cdot,\nabla_{\cdot}(\rho(x)V)\rangle$ on $V^\perp$. Thus for any $x\in O_q$, by \eqref{simplified curv var} we can compute the first derivatives in the right hand side of \eqref{var of Pf} as follows.
\eq{\label{3434 and 4242}
\left.\frac{d}{ds}\right|_{s=0}R^{(s)}_{3434}=3(\rho(x))^2Q_{34}Q_{43},\quad\left.\frac{d}{ds}\right|_{s=0}R^{(s)}_{4242}=3(\rho(x))^2Q_{42}Q_{24},
}

\eq{\label{2323 and 4243}
\left.\frac{d}{ds}\right|_{s=0}R^{(s)}_{2323}=3(\rho(x))^2Q_{23}Q_{32},\quad\left.\frac{d}{ds}\right|_{s=0}R^{(s)}_{4243}=3(\rho(x))^2Q_{42}Q_{34},
}

\eq{\label{3234 and 2324}
\left.\frac{d}{ds}\right|_{s=0}R^{(s)}_{3234}=3(\rho(x))^2Q_{32}Q_{43}~\mathrm{and}~
\left.\frac{d}{ds}\right|_{s=0}R^{(s)}_{2324}=3(\rho(x))^2Q_{23}Q_{42}.
}
Apply \eqref{3434 and 4242}, \eqref{2323 and 4243}, \eqref{3234 and 2324} and the skew-symmetry of $(Q_{ij})_{2\leq i,j\leq 4}$ to \eqref{var of Pf}, we obtain
\eq{\label{var of Pf as var of det}
\left.\frac{d}{ds}\right|_{s=0}\Pf_{g^{(s)}}(x)=&3(\rho(x))^2(R_{1212}Q_{34}Q_{43}+R_{1313}Q_{42}Q_{24}+R_{1414}Q_{23}Q_{32}\\
&\quad\quad\quad\quad-2R_{1213}Q_{42}Q_{34}-2R_{1214}Q_{32}Q_{43}-2R_{1314}Q_{23}Q_{42})\\
=&{-3(\rho(x))^2\left.\frac{d}{ds}\right|_{s=0}\det\matiii{sR_{1212}&Q_{23}+sR_{1213}&Q_{24}+R_{1214}}{Q_{32}+sR_{1312}&sR_{1313}&Q_{34}+sR_{1314}}{Q_{42}+sR_{1412}&Q_{43}+sR_{1413}&sR_{1414}}}\\
=&-3(\rho(x))^2\left.\frac{d}{ds}\right|_{s=0}\det\left[(Q_{(i+1)(j+1)})_{1\leq i,j\leq 3}+s(R_{1(i+1)1(j+1)})_{1\leq i,j\leq3}\right]\\
=&3(\rho(x))^2\left.\frac{d}{ds}\right|_{s=0}\det\left[\langle\nabla_{\cdot}V,\cdot\rangle|_{V^\perp}-sR(\cdot,V,\cdot,V)|_{V^{\perp}}\right].
}
Denoted by $\dvol_g,\dvol_{g^{(s)}}$ the volume form with respect to metrics $g,g^{(s)}$ respectively. Then $\dvol_{g^{(s)}}=\left[1+s(\rho(x))^2\right]^{1/2}\dvol_g$. By the Gauss-Bonnet-Chern formula \eqref{GBC formula} and the above first variation of Pfaffian function \eqref{var of Pf as var of det}, we have
\eq{\label{var of GBC}
0=\left.\frac{d}{ds}\right|_{s=0}(2\pi)^2\chi(M)=&\left.\frac{d}{ds}\right|_{s=0}\int_M\Pf_{g^{(s)}}(x)\dvol_{g^{(s)}}(x)\\
=&\int_M\left.\frac{d}{ds}\right|_{s=0}\Pf_{g^{(s)}}(x)\dvol_{g}(x)\quad\quad(\Pf_g=\Pf_{g^{(0)}}\equiv 0.)\\
=&~3\int_{O_q}(\rho(x))^2\left.\frac{d}{ds}\right|_{s=0}\det\left[\langle\nabla_{\cdot}V,\cdot\rangle|_{V^\perp}-sR(\cdot,V,\cdot,V)|_{V^{\perp}}\right]\dvol_{g}(x).
}
We will show that the integrand in \eqref{var of GBC} is non-negative via the following linear algebra fact.
\begin{fact*}
Let $L,N$ be $3\times 3$ matrices with real entries. Suppose that $L$ is skew-symmetric and $N$ is negative definite. Then we have
$$\left.\frac{d}{ds}\right|_{s=0}\det(L-sN)\geq0.$$
Equality holds if and only of $L=0$.
\end{fact*}
\begin{proof}[Proof of Fact.]
Without loss of generality, we can assume that $N$ is diagonal with negative entries on its diagonal. Since $L$ is skew-symmetric, $L_{ij}=-L_{ji}$ for any $1\leq i,j\leq 3$. Hence
$$\left.\frac{d}{ds}\right|_{s=0}\det(L-sN)=N_{11}L_{23}L_{32}+N_{22}L_{13}L_{31}+N_{33}L_{12}L_{21}\geq 0.$$
Since $N_{11},N_{22},N_{33}<0$, equality holds if and only if $L_{23}=L_{12}=L_{13}=0$, which is equivalent to $L=0$. 
\end{proof}
Back to the proof of Lemma \ref{integrable+totally geo flats}, we apply the fact to the case when $L$ is given by the matrix of $\langle\nabla_\cdot V,\cdot\rangle$ and $N$ is given by the matrix of $R(\cdot,V,\cdot,V)$ and hence obtain
$$0=3\int_{O_q}(\rho(x))^2\left.\frac{d}{ds}\right|_{s=0}\det\left[\langle\nabla_{\cdot}V,\cdot\rangle|_{V^\perp}-sR(\cdot,V,\cdot,V)|_{V^{\perp}}\right]\dvol_{g}(x)\geq 0.$$
Equality holds if and only if $\langle\nabla_\cdot V,\cdot\rangle\equiv0$ in $\mathrm{supp}\rho$. In particular $\langle\nabla_\cdot V,\cdot\rangle=0$ at $q$. By arbitariness of $q\in U_p$, we have $\langle\nabla_\cdot V,\cdot\rangle\equiv 0$ in $U_p$, which proves Lemma \ref{integrable+totally geo flats}.
\end{proof}

\section{Proof of Theorem \ref{main thm}}\label{s4}
If $M$ is not orientable, we can pass to its orientable $2$-fold cover which also has vanishing Euler characteristic. Since Theorem \ref{main thm} is local, 
 the conclusion of Theorem \ref{main thm} for $M$ follows directly from the orientable case. Therefore we assume that $M$ is orientable.

First, we prove that when $\chi(M)=0$, the Ricci curvature must be degenerate somewhere. If not, then the Ricci curvature is negative definite everywhere on $M$. By Lemma \ref{prelim local structure}, there exist a smooth real line bundle $E\subset TM$ such that for any $p\in M$ and $v\in E^1_p:=E\cap T_p^1M$, $\mathrm{Ric}_g(v,v)=R_g/2$. In particular, for any $w,u\in T_pM\setminus\{0\}$ which are perpendicular to $v$, we have 
$$R(w,u,w,u)=0\mathrm{~and~}R(v,u,v,u),R(v,w,v,w)<0.$$
Since $E^1_p$ consists of 2 points, for any connected component $ \widehat E^1$ of $E\cap T^1M$, the natural projection map $\widehat E^1\to M$ is covering map with degree at most 2. Let us equip $\widehat E^1$ with the pullback metric $g_E$ of $g$ via the natural projection map $\widehat E^1\to M$. Then the smooth vector field $V$ on $\widehat E^1$ defined by 
$$V|_{(p,v)}=v\in T_{(p,v)}\widehat E^1\simeq T_pM,~\forall (p,v)\in \widehat E^1$$
is a unit vector field on $(\widehat E^1,g_E)$ such that for any $x\in \widehat E^1$ and any $u,w\in T_x\widehat E^1\setminus\{0\}$ which are perpendicular to $V|_x$, we have 
$$R(w,u,w,u)=0\mathrm{~and~}R(V|_x,u,V|_x,u),R(V|_x,w,V|_x,w)<0.$$
In particular, $\Ric_g(V,V)=R_g/2$. Since $(\widehat E^1,g_E)$ is a Riemannian covering of $M$ with $0=\chi(\widehat E^1)=\chi(M)$ or $2\chi(M)$, in order to prove that there exist some point with degenerate Ricci curvature, we can assume WLOG that $(M,g)=(\widehat E^1,g_E)$.

By Lemma \ref{integrable+totally geo flats} and the fact that $V$ is a unit vector field, we have 
\eq{\label{4.1}\nabla_XV=0,~\forall X\perp V~\mathrm{and}~\nabla_VV\perp V.}

We now consider the vector field $\nabla_VV$, for any point $p\in M$, we choose vector fields $X_1:=V, X_2,X_3,X_4$ such that they form an orthonormal frame when restricted to a neighborhood $U_p$ of $p$. Then on $U_p$, we have
\begin{align*}
\mathrm{div}(\nabla_VV)=&\langle\nabla_V\nabla_VV,V\rangle+\sum_{j=2}^4\langle\nabla_{X_j}\nabla_VV,X_j\rangle\\
=&V\langle\nabla_VV,V\rangle-|\nabla_VV|^2_g+\sum_{j=2}^4\left(R(V,X_j,V,X_j)+\langle\nabla_V\nabla_{X_j}V,X_j\rangle-\langle\nabla_{[V,X_j]}V,X_j\rangle\right)\\
=&-|\nabla_VV|^2_g+\Ric_g(V,V)-\sum_{j=2}^4\langle\nabla_{\nabla_VX_j-\nabla_{X_j}V}V,X_j\rangle\quad\quad\,\,\,\, \,\,\,\,(\mathrm{By}~\eqref{4.1}.)\\
=&-|\nabla_VV|^2_g+\Ric_g(V,V)-\sum_{j=2}^4\langle\nabla_{\nabla_VX_j}V,X_j\rangle\quad\quad\,\,\,\,\,\,\,\,\,\,\,\,\,\,\,\, \,\,\,\,\,\,\,\,(\mathrm{By}~\eqref{4.1}.)\\
=&-|\nabla_VV|^2_g+\Ric_g(V,V)-\sum_{j=2}^4\langle\nabla_{\langle\nabla_VX_j,V\rangle V+(\nabla_VX_j-\langle\nabla_VX_j,V\rangle V)}V,X_j\rangle\\
=&-|\nabla_VV|^2_g+\Ric_g(V,V)-\sum_{j=2}^4\langle\nabla_VX_j,V\rangle\langle\nabla_{V}V,X_j\rangle\,\,\,\,\quad\quad (\mathrm{By}~\eqref{4.1}.)\\
=&-|\nabla_VV|^2_g+\Ric_g(V,V)-\sum_{j=2}^4(V\langle X_j,V\rangle-\langle X_j,\nabla_VV\rangle)\langle\nabla_{V}V,X_j\rangle\\
=&-|\nabla_VV|^2_g+\Ric_g(V,V)+\sum_{j=2}^4\langle\nabla_{V}V,X_j\rangle^2\\
=&-|\nabla_VV|^2_g+\Ric_g(V,V)+|\nabla_VV|^2_g=\Ric_g(V,V)=\frac{R_g}{2}\quad(\mathrm{By}~\eqref{4.1}.)
\end{align*}
By the divergence theorem
$$0=\int_M\mathrm{div}(\nabla_VV)\dvol_g=\frac{1}{2}\int_MR_g\dvol_g.$$
This implies that $R_g\equiv 0$ and hence $M$ is flat, contradictory to the assumption that $\Ric_g$ is non-degenerate everywhere. Hence $\Ric_g$ must be degenerate somewhere.

To prove the local structure result stated in Theorem \ref{main thm}, we first notice that for any $p\in M$ such that $\Ric_g(p)$ is negative definite, by Lemma \ref{integrable+totally geo flats}, there exist some neighborhood $U_p$ and a unit vector field $V$ on $U_p$ such that $V^\perp$ is an integrable distribution whose integral manifolds are totally geodesic and flat. Hence there exist $\psi:U\to U_p$, a smooth diffeomorphism from a connected neighborhood $U$ of $0$ in $\R^4$ onto its image, such that for any $c\in\R$, $\psi(\cdot,\cdot,\cdot,c)$ is an integral manifold of $V^\perp$. Choose a convex open neighborhood $U_p'$ of $p$ which is properly contained in $\psi(U)$, i.e. $\overline{U_p'}\subset\mathrm{int}(\psi(U))$. 
For any $q=\psi(x_1,x_2,x_3,x_4)\in \psi(U)$, we denote by $F_{x_4}:=\{(y_1,y_2,y_3,x_4)|(y_1,y_2,y_3,x_4)\in U\}$ and $\cF_q:=\psi(F_{x_4})$. By the definition of $\psi$, $\cF_q$ is a local totally geodesic and flat leaf containing $q$ tangent to $V^\perp$. Then there exist some $\epsilon>0$ satisfying the following properties:
\begin{itemize}
\item For any $q\in U_p'$, the integral curve $\phi_t(q)$ of $V$ at $q$ is well-defined for $t\in[-\epsilon,\epsilon]$;
\item $\bigcup_{|t|\leq\epsilon}\phi_t(U_p')$ is properly contained in $\psi(U)$, i.e. $\overline{\bigcup_{|t|\leq\epsilon}\phi_t(U_p')}\subset\mathrm{int}(\psi(U))$;
\item The map $(x_1,x_2,x_3,x_4)\to\phi_{x_4}(\psi(x_1,x_2,x_3,0))$ is a diffeomorphism from $(F_0\cap \psi^{-1}(U_p'))\times(-\epsilon,\epsilon)$ onto its image.
\end{itemize}
For any $q\in \psi(U)$, we write $((\psi^{-1})_*V)|_{\psi^{-1}(q)}=(h_1(q),...,h_4(q))\in\R^4\cong T_{\psi^{-1}(q)}\R^4$, where $h_j(\cdot):\psi(U)\to\R$ are smooth functions. Since $(\psi^{-1})_*(V^\perp|_q)=\{(v_1,...,v_4)|v_4=0\}$, we have $h_4(q)$ are strictly positive or negative for any $q\in\psi(U)$. Let $\widehat V=V/h_4$. Then there exist some $0<\widehat\epsilon\leq\epsilon/\left(1+\sup_{|t|\leq\epsilon,q\in \overline{\bigcup_{|t|\leq\epsilon}\phi_t(U_p')}}|h_4(\phi_t(q))|\right)$ satisfying the following properties:  
\begin{itemize}
\item For any $q\in U_p'$, the integral curve $\widehat\phi_t(q)$ of $\widehat V$ at $q$ is well-defined for any $t\in[-\widehat\epsilon,\widehat\epsilon]$;
\item $\bigcup_{|t|\leq\widehat\epsilon}\widehat{\phi}_t(U_p')\subset\bigcup_{|t|\leq\epsilon}\phi_t(U_p')$;
\item The map $(x_1,x_2,x_3,x_4)\to\phi_{x_4}(\widehat\psi(x_1,x_2,x_3,0))$ is a diffeomorphism from $(F_0\cap \psi^{-1}(U_p'))\times(-\widehat\epsilon,\widehat\epsilon)$ onto its image.
\end{itemize}
Let 
$$U_{p,\delta}':=\bigcup_{|t|\leq\delta}\widehat{\phi}_t(U_p'),~\forall\, 0\leq \delta\leq \widehat\epsilon.$$ Then by the definition of $\psi$ and $\widehat V$, for any $q_1,q_2\in U_{p,\widehat\epsilon/3}$ and any $t\in[-\widehat\epsilon/3,\widehat\epsilon/3]$, $\cF_{q_1}=\cF_{q_2}$ if and only if $\cF_{\widehat\phi_t(q_1)}=\cF_{\widehat\phi_t(q_2)}$. In particular, for any $t\in(-\widehat\epsilon/3,\widehat\epsilon/3)$ and any smooth curve $\gamma:[0,1]\to\cF_p\cap U_p'=\psi(F_0\cap\psi^{-1}(U_p'))$, we have $\widehat\phi_t(\gamma([0,1]))\subset \cF_{\widehat\phi_t(p)}$ and hence
\begin{align*}
\frac{d}{dt}\mathrm{length}_g(\widehat\phi_t(\gamma))=&\int_0^1\frac{d}{dt}\|(\widehat\phi_t)_*\dot\gamma(s)\|_gds\\
=&\int_0^1\frac{\langle\nabla_{\widehat V}(\widehat\phi_t)_*\dot\gamma(s),(\widehat\phi_t)_*\dot\gamma(s)\rangle}{\|(\widehat\phi_t)_*\dot\gamma(s)\|_g}ds\\
=&\int_0^1\frac{\langle\nabla_{(\widehat\phi_t)_*\dot\gamma(s)}\widehat V,(\widehat\phi_t)_*\dot\gamma(s)\rangle}{\|(\widehat\phi_t)_*\dot\gamma(s)\|_g}ds\quad\quad\left(\begin{subarray}
~ \mathrm{Differentiation~along~}t\mathrm{~and~}s\\
\mathrm{commutes~for~}\widehat\phi_t(\gamma(s)).
\end{subarray}\right)\\
=&\int_0^1\frac{\langle\nabla_{(\widehat\phi_t)_*\dot\gamma(s)} V,(\widehat\phi_t)_*\dot\gamma(s)\rangle}{h_4(\widehat\phi_t(\gamma(s)))\|(\widehat\phi_t)_*\dot\gamma(s)\|_g}ds\quad\left(\begin{subarray}
~ V\perp(\phi_t)_*\dot\gamma(s)\mathrm{~since}\\
\widehat\phi_t(\gamma([0,1]))\subset \cF_{\widehat\phi_t(p)}.
\end{subarray}\right)\\
=&\,0.\quad\quad\quad\quad\quad\quad\quad\quad\quad\quad\quad\quad\quad\quad \left(\begin{subarray}
~ \mathrm{By~Lemma~\ref{integrable+totally geo flats}}\mathrm{and~}V\perp(\widehat\phi_t)_*\dot\gamma(s),\\
\mathrm{the~integrand~of~the~above~is~zero.}
\end{subarray}\right)
\end{align*}
As a corollary, $\widehat\phi_t|_{\cF_p\cap U_p'}:(\cF_p\cap U_p',g|_{\cF_p\cap U_p'})\to(\cF_{\widehat\phi_t(p)}\cap\widehat\phi_t(U_p'),g|_{\cF_{\widehat\phi_t(p)}\cap\widehat\phi_t(U_p')})$ is an isometry between convex open subsets of 3-dimensional Euclidean spaces (embedded into $M$ as local totally geodesic and flat leaves). 

Let $v_1,v_2,v_3,v_4=V|_p$ be an orthonormal basis for $T_pM$. Consider the map $\Psi:\R^3\to M$ such that $\Psi(x_1,x_2,x_3)=\exp_p(x_1v_1+x_2v_2+x_3v_3)$. We then define the following local coordinate map near $p$:
$$\widehat\Psi:\Psi^{-1}(\cF_p\cap U_p')\times(-\widehat\epsilon/3,\widehat\epsilon/3)\to M,~\widehat\Psi(x_1,x_2,x_3,x_4):=\widehat\phi_{x_4}(\Psi(x_1,x_2,x_3)).$$
Then $\widehat\Psi$ is a smooth diffeomorphism onto its image. Moreover, 
$$\widehat\Psi^{*}g=dx_1^2+dx_2^2+dx_3^2+(f(x_1,x_2,x_3,x_4))^2dx_4^2$$ 
for some positive smooth function $f(x_1,x_2,x_3,x_4)$. (This is because $\widehat\Psi_*(\del/\del x_j)$ are perpendicular to each other for any $1\leq j\leq 4$ and $\widehat\Psi_*(\del/\del x_j)$ are unit vectors for any $1\leq j\leq 3$.) Strict convexity of $f(\cdot,\cdot,\cdot,x_4)$ follows directly from \eqref{convexity} and the fact $R(w,V,w,V)<0$ for any $0\neq w\in V^\perp$.

\section{Simplicial volume and Euler characteristics}\label{s5}
Similar to the Euler characteristics, another topological invariant which interplays nicely with the geometry of a manifold, especially for nonpositively curved manifolds, is the \emph{simplicial volume}. Let $M$ be a closed, connected, oriented topological $n$-manifold. The simplicial volume of $M$ is defined as
\[||M||:=\inf\left\{\sum_{i=1}^\ell |a_i|:\left[\sum_{i=1}^\ell a_i \sigma_i\right]=[M]\right\}\]
where the infimum is taken over all real singular cycles $\sum_{i=1}^\ell a_i \sigma_i$ representing the fundamental class $[M]\in H_n(M;\mathbb R)$. The following two conjectures are due to Gromov \cite{Sav} and \cite[p. 232]{Gromov}.

\begin{conjecture}\label{conj:Ricci}
	If $M$ admits a Riemannian metric with nonpositive sectional curvature and negative definite Ricci curvature, then $||M||>0$.
\end{conjecture}

\begin{conjecture}\label{conj:Euler}
	If $M$ is aspherical and $||M||=0$, then $\chi(M)=0$.
\end{conjecture}

We refer to \cite{LoehEtAl22, LoehRaptis23} for a vivid discussion on some of the progress and possible approaches regarding Conjecture \ref{conj:Euler}. Partial results on Conjecture \ref{conj:Ricci} include the case of locally symmetric spaces of noncompact type \cite{LS06, Bu} and rank one manifolds with a stronger Ricci condition \cite{CW19}. It is also observed \cite{CW20} that using a local straightening technique, the Ricci condition in the latter case need only be satisfied at one point, thus a stronger conjecture was proposed as follows \cite[Conjecture 4.1]{CW20},

\begin{conjecture}\label{conj:Ricci-strong}
	If $M$ admits a Riemannian metric with nonpositive sectional curvature everywhere and negative definite Ricci curvature at a point, then $||M||>0$.
\end{conjecture}

It is clear that Conjecture \ref{conj:Ricci-strong} implies Conjecture \ref{conj:Ricci}. Our Theorem \ref{main thm} further shows that when restricting to nonpositively curved $4$-manifolds, Conjecture \ref{conj:Euler} in facts sits in between Conjecture \ref{conj:Ricci-strong} and Conjecture \ref{conj:Ricci}.

\begin{corollary}
	Let $M$ be a nonpositively curved $4$-manifolds. Then
	\[\textrm{Conjecture } \ref{conj:Ricci-strong}\implies \textrm{Conjecture } \ref{conj:Euler}\implies \textrm{Conjecture } \ref{conj:Ricci}.\]
\end{corollary}

\begin{proof}
``Conjecture \ref{conj:Ricci-strong} $\implies$ Conjecture \ref{conj:Euler}'': We prove the contrapositive statement. Suppose $\chi(M)>0$, the Gauss-Bonnet-Chern's theorem implies $\Pf(p)>0$ for some $p\in M$, then Lemma \ref{lem:zeroRic_implies_zeroPf} shows that $\Ric$ is nondegenerate at $p$, thus by Conjecture \ref{conj:Ricci-strong}, we have $||M||>0$. This proves Conjecture \ref{conj:Euler}.

``Conjecture \ref{conj:Euler}$\implies$ Conjecture \ref{conj:Ricci}'': We prove the contrapositive statement. Suppose $||M||=0$, then by Conjecture \ref{conj:Euler} we have $\chi(M)=0$. Thus by Theorem \ref{main thm}, Ricci curvature must be degenerate somewhere. This proves Conjecture \ref{conj:Ricci}.
\end{proof}

It is natural to ask whether the converse of Conjecture \ref{conj:Euler} holds for nonpositively curved even dimensional manifolds, that is, does $\chi(M)=0$ imply $||M||=0$. This is obviously true for surfaces ($n=2$), but in high dimensions ($n\geq 6$) it is false. For example, one can take $M$ to be the product of two closed hyperbolic $3$-manifolds, then $\chi(M)=0$ but $||M||>0$. It remains mysterious in dimension $4$. In view of Conjecture \ref{conj:Euler}, we propose the following stronger conjecture.

\begin{conjecture}\label{conj:dim4}
	Let $M$ be a closed nonpositively curved $4$-manifold. Then $||M||=0$ if and only if $\chi(M)=0$.
\end{conjecture}

If we further restrict to real analytic nonpositively curved $4$-manifolds, then Corollary \ref{main cor} shows that $\chi(M)=0$ if and only if $\widetilde M$ has a nontrivial $\R$-factor, in which case we know $||M||=0$. Therefore the converse of Conjecture \ref{conj:Euler} holds thus Conjecture \ref{conj:dim4} and Conjecture \ref{conj:Euler} are equivalent.

\def\cprime{$'$}
\providecommand{\bysame}{\leavevmode\hbox to3em{\hrulefill}\thinspace}
\providecommand{\MR}{\relax\ifhmode\unskip\space\fi MR }
\providecommand{\MRhref}[2]{%
	\href{http://www.ams.org/mathscinet-getitem?mr=#1}{#2}
}
\providecommand{\href}[2]{#2}

\end{document}